\documentclass[11pt]{amsart}

\usepackage[active]{srcltx}

\usepackage[pdftex,left=2.65cm,right=2.65cm,top=2.74cm,bottom=2.74cm]{geometry}

\usepackage[colorlinks=true,urlcolor=blue,
citecolor=red,linkcolor=blue,linktocpage,pdfpagelabels,
bookmarksnumbered,bookmarksopen]{hyperref}
\usepackage{amsmath,amssymb}
\usepackage{amstext}
\usepackage{enumerate}
\usepackage{epsfig}
\usepackage{color}
\usepackage{esint}

\newtheorem{theorem}{Theorem}[section]

\newtheorem{proposition}{Proposition}[section]

\newtheorem{lemma}[proposition]{Lemma}

\theoremstyle{definition}

\numberwithin{equation}{section}

\newcommand{\R}{{\mathbb R}}
\newcommand{\Z}{{\mathbb Z}}

\newcommand{\I}{{\mathbb I}}

\allowdisplaybreaks

\title{On Coron's problem for the $p$-Laplacian}
\author{Carlo Mercuri}
\author{Berardino Sciunzi}
\author{Marco Squassina}

\thanks{The second and third authors were partially  supported by the MIUR project:
   ``Variational and Topological Methods in the Study of Nonlinear Phenomena''.}
\subjclass[2000]{35J92, 58E05, 54D30}
\keywords{Coron's problem, quasi-linear equations, critical exponent}

\begin{document}

\begin{abstract}
We prove that the critical problem for the $p$-Laplacian operator admits
a nontrivial solution in annular shaped domains with
sufficiently small inner hole. This extends
Coron's problem to a class of quasilinear problems.
\end{abstract}

\maketitle

\section{Introduction}
\noindent
We want to extend the classical result of Coron \cite{coron}. Consider the problem
\begin{equation}\label{maineq}
\left\{
\begin{array}{lll}
-\Delta_p u = |u|^{p^*-2}u  \quad &\mathrm{in}& \,\,\Omega   \\
  \,\,\, u=0 &\mathrm{on}& \,\, \partial \Omega,
\end{array}
\right.
\end{equation}
where  $\Omega$ is a smooth bounded domain in $\R^N,$ $1<p<N,$ $p^*:=N p/(N-p)$ is the critical Sobolev exponent, $\Delta_p u:={\rm div}(|\nabla u|^{p-2}\nabla u)$ is the $p$-Laplace operator.
 Solutions on the whole space will be considered in  $$\mathcal D^{1,p}(\mathbb R^N):=\{u\in L^{p^*}(\R^N):\nabla u\in L^p(\mathbb R^N;\mathbb R^N)\}$$
endowed with the norm $$\|u\| :=\|\nabla u\|_{L^p(\mathbb R^N)}.$$
We denote by $ W_0^{1,p}(\Omega)$ the closure of $C^\infty_c(\Omega)$ in $\mathcal D^{1,p}(\R^N)$ and define on $W_0^{1,p}(\Omega)$ the functional
\begin{equation*}
J(u):=\frac{1}{p}\int_{\Omega}|\nabla u|^p dx-\frac{1}{p^*} \int_{\Omega}|u|^{p^*}dx.
\end{equation*}
As it is well-known in tackling problem (\ref{maineq}) with variational techniques, the main difficulty is due to the fact that the embedding $ W_0^{1,p}(\Omega)\subset L^{p^*}(\Omega)$ is not compact. 
We refer to \cite{struwe} for a sample of the extensive literature on semi-linear problems involving the critical Sobolev exponent, largely inspired by the pioneering paper of Brezis and Nirenberg \cite{breni}. 
We also define 
\begin{equation*}
S:=\inf\Big\{\int_{\R^N}|\nabla u|^p dx,\, u\in {\mathcal D}^{1,p}(\R^N)\, : \, \int_{\R^N}|u|^{p^*} dx =1\Big\}
\end{equation*}
the best Sobolev constant, attained by nowhere zero functions in $\R^N,$ see e.g. \cite{talenti}.
Equivalently
\begin{equation}
\label{SSS}
S\,=\,\inf_{\underset{u\neq 0}{u\in{\mathcal D}^{1,p}(\mathbb{R}^N)} }\dfrac{\int_{\mathbb{R}^N}|\nabla\,u|^pdx}{\left(\int_{\mathbb{R}^N}\,|u|^{p^*}dx\right)^{\frac{p}{p^*}}},
\end{equation}
where by a simple scaling argument the infimum remains unchanged if taken
on competing functions supported in an arbitrary subdomain of $\R^N$.
In light of the Pohozaev identity obtained by Guedda and Veron \cite[Corollary 3.1]{PohoGV}, 
we know that problem \eqref{maineq} does not admit positive solutions on a strictly star-shaped domain.
\vskip4pt
\noindent
The main result of the paper is the following
\begin{theorem}\label{coronn}
Let $2N/(N+2)<p\leq 2,$ $x_0\in \R^N$ and radii $R_2>R_1>0$ such that
\begin{equation}
\label{domain-ass}
\{R_1\leq |x-x_0|\leq R_2\}\subset \Omega,\qquad
\{|x-x_0|\leq R_1\}\not\subset \overline\Omega.
\end{equation}
Then problem \eqref{maineq} admits a positive solution
for $R_2/R_1$  sufficiently large.
\end{theorem}

\noindent
Theorem \ref{coronn} is, mainly, a consequence of Lemma \ref{PS}, in which the compactness result \cite[Theorem 1.2]{MeWi} and the symmetry result of \cite{ClassDino} play a key role. There are several difficulties arising in the present quasilinear setting which are partially highlighted in Lemma \ref{PS}, which make the proof more delicate than for dealing with the semilinear case $p=2.$ One of those is the fact that the classification of all positive solutions of the critical problem in $\R^N$ is not yet available for all $p\in(1,N).$ We observe that an extension of Lemma \ref{PS} to a broader range of $p$ would immediately yield an extension of Theorem \ref{coronn}. We conjecture that the symmetry result of \cite{ClassDino} and hence Lemma \ref{PS} and Theorem \ref{coronn} hold for all values of $p\in (1,N).$ Another open problem, arising in the proof of Lemma \ref{PS}, is the nonexistence of sign-changing solutions of the critical problem in the half-space for $p\neq 2$. Such a limiting problem arises because of the boundary of $\Omega.$ We show that in fact only the nonexistence result of the positive solutions of the critical problem in the half-space \cite[Theorem 1.1]{MeWi} is needed.
In the case $N=2$ Theorem \ref{coronn} holds for all $1<p<2,$ which is the desired range for a $p$-Laplacian extension of the classical result of Coron.  We point out that Theorem \ref{coronn} extends \cite[ Theorem 1.1]{MePa}, where problem (\ref{maineq}) had been studied assuming that $\Omega$ is invariant under the action of a closed subgroup of $O(N)$.
It is an open problem whether \eqref{maineq} has nontrivial solutions when a $\Z_2$-homology group of $\Omega$ is nontrivial. This is the case
for $p=2$, see the celebrated analysis done in \cite{bahricoron}. In several contributions dealing with the semi-linear case $p=2$, see e.g.
\cite{dancer,ding,passaseo}, it is shown that the existence of a nontrivial solution
is possible also in contractible domains, hence conditions on the homology of $\Omega$ are not necessary for problem \ref{maineq} to have solutions. A very well-known and challenging problem, even in the case $p=2,$ would be to exploit the combined effect of both the topology and the geometry of $\Omega$ in order to characterize the existence of a positive solution to problem \eqref{maineq}.


\section{Proof of Theorem~\ref{coronn}}

\noindent
In this section we prove Theorem~\ref{coronn}.

\subsection{Palais-Smale condition}
\noindent
We define $\R^N_+:=\{x\in\R^N:x_N>0\}$ and denote by $\mathcal D^{1,p}_0(\R^N_+)$ the closure of $C^\infty_c(\R^N_+)$ in $\mathcal D^{1,p}(\R^N)$ after extending by zero on $\R^N \setminus \R^N_+. $
\begin{lemma}\label{sign}
Let $u\in W^{1,p}_0(\Omega)$ be a sign-changing solution to \eqref{maineq}. Then $J(u)\geq 2S^{N/p}/N$. Moreover, the same conclusion
holds for the sign-changing solutions of $-\Delta_p u = |u|^{p^*-2}u$ in $\mathcal D^{1,p}(\R^N)$ or in $\mathcal D^{1,p}_0(\R^N_+).$
\end{lemma}
\begin{proof}
If $u\in W^{1,p}_0(\Omega)$ is a sign-changing solution to \eqref{maineq}, then $u^\pm\in  W^{1,p}_0(\Omega)\setminus\{0\}$ and
by testing  \eqref{maineq} with $u^\pm$ yields
$$
\int_{\Omega} |\nabla u^+|^pdx=\int_{\Omega} |u^+|^{p^*}dx,\qquad
\int_{\Omega} |\nabla u^-|^pdx=\int_{\Omega} |u^-|^{p^*}dx.
$$
In turn, using the definition of \eqref{SSS}, we obtain
\begin{equation*}
J(u)=J(u^+)+J(u^-)=\frac{1}{N}\|u^+\|_{p^*}^{p^*}+\frac{1}{N}\|u^-\|_{p^*}^{p^*}\geq  2S^{N/p}/N,
\end{equation*}
concluding the proof. The same argument works for the problem on $\R^N$ and on $\R^N_+$.
\end{proof}

\begin{lemma}\label{PS}
Assume that $2N/(N+2)<p\leq 2.$ Then
$J$ satisfies the Palais-Smale condition
for all $c\in (S^{N/p}/N,2S^{N/p}/N).$
\end{lemma}
\begin{proof}
Assume that for some $c\in(S^{N/p}/N,2S^{N/p}/N),$ $(u_n)\in W^{1,p}_0(\Omega)$ is such that
$J(u_n)\to c$, and $J'(u_n)\to 0$ in $W^{-1,p'}(\Omega).$
We define on
$\mathcal D^{1,p}(\R^N)$
$$
J_\infty(u):=\int_{\R^N}\frac{|\nabla u|^{p}}{p}dx-\int_{\R^N}\frac{|u|^{p^*}}{p^*}dx.
$$
On $\mathcal D^{1,p}_0(\R^N_+)$ we define the same functional $J_\infty$ extending by zero on $\R^N \setminus \R^N_+. $ \newline
By applying \cite[proof of Theorem 1.2]{MeWi}, which extends \cite{struwe-art}, passing if necessary to a subsequence,
we can infer that there exists a (possibly trivial)
solution $v_0\in W^{1,p}_{0}(\Omega)$ of

$$-\Delta_p u =  |u|^{p^*-2}u \quad \textrm{in}\quad \Omega, $$
$k\in \mathbb N \cup \{0\}$, nontrivial solutions $\{v_1,...,v_k\}$ of
$$
-\Delta_p u =  |u|^{p^*-2}u \quad \textrm{in}\quad H_i,\quad i\in\{0,1,...k\},$$
where $H_i$ is either $\R^N$ or (up to rotation and translation) $\R^N_+,$ with either $v_i \in \mathcal D^{1,p}(\R^N)$ or (respectively)  $v_i \in \mathcal D^{1,p}_0(\R^N_+),$  and  there exist $k$ sequences $\{y^i_n\}_n \subset \bar{\Omega}$ and $\{\lambda^i_n\}_n \subset \R_+,$  satisfying
$$
\frac{1}{\lambda^i_n}\, \textrm{dist} \, (y^i_n,\partial\Omega)\rightarrow \infty , \,\quad n\rightarrow \infty,
$$
if $H_i\equiv\R^N$
or
$$
\frac{1}{\lambda^i_n}\, \textrm{dist} \, (y^i_n,\partial\Omega)< \infty , \,\quad n\rightarrow \infty,
$$
if (up to rotation and translation) $H_i\equiv\R^N_+,$ and
$$
\|u_n-v_0-\sum^k_{i=1}(\lambda^i_n)^{(p-N)/p}v_i ((\cdot-y^i_n)/\lambda^i_n)\|\rightarrow 0, \quad n\rightarrow \infty,
$$
$$
\|u_n\|^p\rightarrow \sum^k_{i=0}\|v_i\|^p, \quad n\rightarrow \infty,
$$
\begin{equation}
\label{levels}
J(v_0)+\sum^k_{i=1}J_\infty (v_i)=c.
\end{equation}
The restriction on the levels $c$ and Lemma \ref{sign} immediately yields the bound $k\leq 1$. If $k=0$ compactness holds and we are done.
If instead $k=1$, we have two cases, namely $v_0\equiv 0$ or $v_0\not\equiv 0$. If $v_0\not\equiv 0$, since
$$
J(v_0)\geq S^{N/p}/N,\qquad J_\infty(v_1)\geq S^{N/p}/N,
$$
(actually $J(v_0)> S^{N/p}/N,$ as the Sobolev constant is never achieved on bounded domains) we obtain a contradiction by combining \eqref{levels} with the assumption $c<2S^{N/p}/N$. If, instead, $v_0\equiv 0$, then
formula \eqref{levels} reduces to $J(v_1)=c$. Again by Lemma \ref{sign} $v_1$ does not change sign and by the nonexistence result \cite[Theorem 1.1]{MeWi} $H_1\equiv\R^N,$ namely $v_1\in \mathcal D^{1,p}(\R^N)$
solves
\begin{eqnarray}
-\Delta_p u = u^{p^*-1} &\textrm{in}& \R^N, \\
u > 0 & \textrm{in} & \R^N. \nonumber
\end{eqnarray}
Now, by the symmetry result of \cite[Theorem 2.1]{ClassDino}, which holds in the range $2N/(N+2)<p\leq 2$, $v_1$
is radially symmetric about some point and, in turn, by using  \cite[Theorem 2.1(ii)]{ClassGV} (see also \cite{bidaut}), after translation in the origin,
for a suitable value of $a>0$ $v_1$ is a Talenti function
$$
v_1(x)= \Big(Na\Big(\frac{N-p}{p-1}\Big)^{p-1}\Big)^{(N-p)/p^2}(a+|x|^{p/(p-1)})^{(p-N)/p},
$$
whose associated energy is $c=J_\infty(v_1)=S^{N/p}/N$ \cite{talenti}, since $v_1$ achieves the best Sobolev constant $S$. This
is a contradiction again, since $c>S^{N/p}/N.$
 This concludes the proof.
\end{proof}

\subsection{Proof of Theorem \ref{coronn} concluded.}
\noindent
Let $R_1,R_2$ be the radii of the annulus as in the statement of Theorem \ref{coronn}. As observed in \cite{coron,struwe}, without loss of generality, we may assume that $x_0=0$, $R_1=1/(4R)$ and $R_2=4R$ where $R>0$ will be chosen sufficiently large. Let us set $\Sigma:=\{x\in\R^N:|x|=1\}$ and consider the family of functions
$$
u^\sigma_t(x):=\left[\frac{1-t}{(1-t)^p+|x-t\sigma|^{\frac{p}{p-1}}}\right]^{\frac{N-p}{p}}\in {\mathcal D}^{1,p}(\R^N),
\,\,\quad \text{for $\sigma\in\Sigma$ and $t\in [0,1)$}.
$$
Moreover, let us now consider a function $\varphi\in C^\infty_c(\Omega)$ be such that $0\leq\varphi\leq 1$ on $\Omega$, $\varphi=1$ on $\{1/2<|x|<2\}$
and $\varphi=0$ outside $\{1/4<|x|<4\}$, then define
$$
\varphi_R(x):=
\begin{cases}
\varphi(Rx) & \text{on $0\leq |x|<\frac{1}{R}$}, \\
1 & \text{on $\frac{1}{R}\leq |x|<R$}, \\
\varphi(x/R) & \text{on $|x|\geq R$}.
\end{cases}
$$
Finally, let us set
$$
w^\sigma_t(x):=u^\sigma_t(x)\varphi_R(x)\in W^{1,p}_0(\Omega),
\quad
w_0(x):=u_0(x)\varphi_R(x),
\quad
u_0(x):=\Big[\frac{1}{1+|x|^{\frac{p}{p-1}}}\Big]^{\frac{N-p}{p}}.
$$
Then, we have the following
\begin{lemma}
\label{stime-unif}
For $\sigma\in\Sigma$ and $t\in [0,1)$,
$\|u^\sigma_t\|=\|u_0\|$, $\|u^\sigma_t\|_{p^*}=\|u_0\|_{p^*}$ and $\|u^\sigma_t\|^p=S\|u^\sigma_t\|_{p^*}^p$.
Furthermore, there holds
$$
\lim_{R\to\infty}\sup_{\sigma\in\Sigma, t\in [0,1)}\|w^\sigma_t-u^\sigma_t\|=0.
$$
\end{lemma}
\begin{proof}
The first properties of $u^\sigma_t$ follow by \cite{talenti}. In the following $C$ will denote
a generic positive constant, independent of
$\sigma\in\Sigma$ and $t\in [0,1)$, which may vary from line to line.
We have the inequality
\begin{equation*}
\int_{\R^N}|\nabla (w^\sigma_t-u^\sigma_t)|^pdx\leq C\sum_{i=1}^4 \I_i,
\end{equation*}
where we have set
\begin{align*}
& \I_1:=\int_{\R^N\setminus B_{2R}} |\nabla u^\sigma_t|^pdx, \\
& \I_2:=\int_{B_{(2R)^{-1}}} |\nabla u^\sigma_t|^pdx, \\
& \I_3:=\frac{1}{R^p}\int_{B_{4R}\setminus B_{2R}}|u^\sigma_t|^pdx, \\
& \I_4:= R^p\int_{B_{(2R)^{-1}}}|u^\sigma_t|^p dx.
\end{align*}
Taking into account that
$$
|\nabla u^\sigma_t(x)|\leq\frac{C}{((1-t)^p+|x-t\sigma|^{\frac{p}{p-1}})^{\frac{N}{p}}}\leq C \quad\text{$|x|\leq\frac{1}{2}$},\qquad
|\nabla u^\sigma_t(x)|\leq\frac{C}{|x|^{\frac{N-1}{p-1}}}\quad\text{$|x|\geq 2$},
$$
we obtain
\begin{align*}
\I_1&=\int_{\R^N\setminus B_{2R}} |\nabla u^\sigma_t|^pdx\leq C\int_{\R^N\setminus B_{2R}}
\frac{1}{|x|^{\frac{p(N-1)}{p-1}}}dx\leq\frac{C}{R^{\frac{N-p}{p-1}}},  \\
\I_2 &=\int_{B_{(2R)^{-1}}} |\nabla u^\sigma_t|^pdx\leq
C\int_{B_{(2R)^{-1}}}dx\leq \frac{C}{R^N}.
\end{align*}
Moreover, we have
\begin{align*}
\I_3 &=  \frac{1}{R^p}\int_{B_{4R}\setminus B_{2R}}
\Big[\frac{1-t}{(1-t)^p+|x-t\sigma|^{\frac{p}{p-1}}}\Big]^{N-p}dx
 \leq \frac{C}{R^p}
\int_{B_{4R}\setminus B_{2R}}
\frac{1}{|x|^{\frac{p(N-p)}{p-1}}}dx
\leq\frac{C}{R^{\frac{N-p}{p-1}}},  \\
\I_4 &=  R^p\int_{B_{(2R)^{-1}}}
\Big[\frac{1-t}{(1-t)^p+|x-t\sigma|^{\frac{p}{p-1}}}\Big]^{N-p}dx
 \leq R^pC \int_{B_{(2R)^{-1}}}dx\leq\frac{C}{R^{N-p}}.
\end{align*}
This concludes the proof.
\end{proof}

\noindent
Let us now define
\begin{equation}
S(u):=\frac{\|\nabla u\|^p}{\|u\|_{L^{p^*}(\R^N)}^p},\qquad u\in {\mathcal D}^{1,p}(\R^N)\setminus\left\{0\right\},
\end{equation}
with the understanding that
\begin{equation}
S(u;\Omega)=\frac{\|\nabla u\|^p_{L^p(\Omega)}}{\|u\|_{L^{p^*}(\Omega)}^p},
\qquad u\in W^{1,p}_0(\Omega)\setminus\left\{0\right\},
\end{equation}
after extending by zero outside $\Omega$.

\vskip4pt
\noindent
As a consequence of Lemma \ref{stime-unif}, we have the following
\begin{lemma}
\label{estimates}
If $v^\sigma_t(x):= \|w^\sigma_t\|_{L^{p^*}(\R^N)}^{-1} w^\sigma_t(x)$
and $v_0(x)=\|w_0\|_{L^{p^*}(\R^N)}^{-1} w_0(x)$, then
$$
\lim_{R\to\infty} S(v^\sigma_t;\Omega)=S(u^\sigma_t)=S,
$$
uniformly with respect to $\sigma\in \Sigma$ and $t\in [0,1)$.
\end{lemma}

\noindent
We observe that $J$ satisfies the Palais-Smale condition between the levels $S^{N/p}/N$ and $2S^{N/p}/N$.
Therefore, as it can be readily verified, the functional $S(\cdot; \Omega)$, constrained to
$$
{\mathcal M}=\{u\in W^{1,p}_0(\Omega):\|u\|_{p^*}^{p^*}=1\},
$$
satisfies the Palais-Smale condition between $S$ and $\varpi S$, for some $\varpi>1$
depending upon $p$ and $N$.
Then, taking Lemma~\ref{estimates} into account, and assuming by contradiction that the problem does not
admit any positive solution, by arguing exactly as in \cite[pp.191-193]{struwe} one proves Theorem \ref{coronn}
by performing a well-established deformation argument on $S(\cdot;\Omega)$ as restricted to   ${\mathcal M}$,
yielding a contradiction with the geometrical properties \eqref{domain-ass} of $\Omega$.
We point out that under our assumption $2N/(N+2)<p$, it follows $p^*>2$ so that ${\mathcal M}$ is a $C^{1,1}$ smooth manifold. \qed

\section*{Acknowledgements} C.M. would like to thank Prof.\ Abbas Bahri for various discussions at Rutgers University on noncompact problems.

\bigskip
\bigskip

\end{document}